\documentclass[preprint,12pt]{elsarticle}

\usepackage{t1enc}

\usepackage[latin2]{inputenc}
\usepackage{amssymb}
\usepackage{amsmath,amsthm}

\usepackage{mathrsfs}
\usepackage{color}
\usepackage{latexsym}
\usepackage{epsfig}
\usepackage[a4paper,left=4cm,right=3cm,top=3cm,bottom=3cm]{geometry}

\journal{Stochastic Processes and their Applications}

\newtheorem{theorem}{Theorem}

\newtheorem{corollary}{Corollary}

\newtheorem{lemma}{Lemma}

\newtheorem{proposition}{Proposition}
\newtheorem{remark}{Remark}

\newcommand{\eps}{\varepsilon}

\renewcommand{\a}{\alpha}
\renewcommand{\phi}{\varphi}
\renewcommand{\epsilon}{\varepsilon}

\newcommand{\E}{\text{E}}
\renewcommand{\P}{\text{Pr}}

\begin{document}

\begin{frontmatter}



\title{A sharp adaptive confidence ball for self-similar functions}


\author[label1]{Richard Nickl}
\author[label2,label3]{Botond Szab\'o\corref{cor1}}
\address[label1]{University of Cambridge}
\address[label2]{University of Amsterdam}
\address[label3]{Budapest University of Technology and Economics}

\cortext[cor1]{The research was partially supported by the Netherlands Organization for Scientific Research NWO}

\begin{abstract}
In the nonparametric Gaussian sequence space model an $\ell^2$-confidence ball $C_n$ is constructed that adapts to unknown smoothness and Sobolev-norm of the infinite-dimensional parameter to be estimated. The confidence ball has exact and honest asymptotic coverage over appropriately defined `self-similar' parameter spaces. It is shown by information-theoretic methods that this `self-similarity' condition is weakest possible.

\end{abstract}


\end{frontmatter}



\section{Introduction}

Successful statistical methodology in high-dimensional and nonparametric models gives rise, either by construction or implicitly, to statistical procedures that \textit{adapt to unknown properties} of the parameter, such as smoothness or sparsity. It is well-known by now (\cite{Low}, \cite{JudLam}, \cite{CL04}, \cite{RobVaart}, \cite{GW08}, \cite{HN11}, \cite{Bull}, \cite{BullNickl}, \cite{NvdG13}, \cite{C13}) that such adaptive procedures cannot straightforwardly be used for uncertainty quantification. Particularly, and unlike in the classical parametric situation, adaptive estimators do \textit{not} automatically suggest valid confidence sets for natural high- or infinite dimensional parameters. Rather, some additional constraints on the parameter space have to be introduced. 

In nonparametric models one such constraint that is naturally compatible with the desired adaptation properties has been studied in \cite{GineNickl}, \cite{HN11}, \cite{Bull}, \cite{SzVZ2}, \cite{CCK14} -- the term `self-similarity assumption' has been associated with this condition, for reasons that will become apparent below. Except for \cite{SzVZ2}, the above references have studied such parameter constraints in the `$L^\infty$-setting' of confidence \textit{bands}, pertaining to the uniform-norm as a statistical loss function. The situation in the `$L^2$-setting' -- where the risk function is induced by the more common integrated squared loss -- is in principle more favourable (see \cite{JudLam}, \cite{CL06}, \cite{RobVaart}, \cite{BullNickl}, \cite{CLM14}, {\cite{SzVZ3}}), and for certain ranges of parameter spaces such `self-similarity' conditions are simply not necessary. However, as will be explained below, for the most meaningful adaptation problems that range over a full scale of Sobolev spaces with possibly unbounded Sobolev-norm of the function to be estimated, the situation becomes more delicate and `self-similarity conditions' are relevant again. 

In the present article we consider the basic nonparametric sequence space model and provide minimal $\ell^2$-type self-similarity constraints on a Sobolev - parameter space that cannot be improved upon from an information theoretic point of view. We also show that an easy to construct, asymptotically exact, confidence ball based on the idea of unbiased risk estimation performs optimally under such constraints. In contrast to most constructions in the literature, no `under-smoothing' is necessary, and the confidence set adapts to minimax rate of convergence and radius constant. 

The interest in this problem is partly triggered by recent progress on the understanding of the frequentist properties of Bayesian uncertainty quantification methods in \cite{SzVZ2}, where $L^2$-type self-similarity conditions have been employed successfully. Combined with some arguments of \cite{SzVZ3} our results imply that natural nonparametric Bayes approaches based on Gaussian priors with hierarchical or {maximum marginal likelihood} empirical Bayes prior specification of the smoothness parameter do \textit{not} achieve the information theoretic limits of uncertainty quantification.

As usual our ideas and techniques carry over from the sequence space model to more common nonparametric regression and density estimation problems, both constructively by virtue of the $L^2 \sim \ell^2$ isometry of the loss functions, and more fundamentally through asymptotic equivalence theory for statistical experiments.

\section{Main results}\label{sec: main}

Consider observations $Y=(y_k: k \in \mathbb N)$ in the Gaussian sequence space model
\begin{align}
y_{k} = f_{k} + \frac{1}{\sqrt n} g_{k}, ~~g_{k} {\stackrel{i.i.d.}{\sim}} N(0,1), ~ k \in \mathbb N,\label{model}
\end{align}
and write ${\Pr}_f$ {or ${\Pr}_f^{(n)}$} for the law of $(y_k: k \in \mathbb N)$. The symbol $\E_f$ {or $\E_f^{(n)}$} denotes expectation under the law $\Pr_f$. Let us assume that the unknown sequence of interest $f=(f_{k})\in\ell^2$ belongs to a Sobolev ball, that is, an ellipsoid in $\ell^2$ of the form
\begin{align*}
S^{s}(B)=\{f\in\ell^2: \|f\|_{s,2} \le B\},~~s>0, ~B>0,
\end{align*}
where the Sobolev norm is given by $$\|f\|_{s,2}^2=\sum_{k=1}^{\infty} f_k^2 k^{2s}.$$ Note that $\|\cdot\|_2 \equiv \|\cdot\|_{0,2}$ is the usual $\ell^2$-norm. 

\smallskip

The parameters $B,s>0$ are typically not available a priori, and the challenge arises to adapt to their unknown values in a data-driven way. We will consider adaptation to `smoothness degrees' $s$ in any fixed window $[s_{\min},s_{\max}]$, and to the `radius' $B \in [b, \infty)$. Here $0<s_{\min}<s_{\max}<\infty$ are fixed and known parameters whereas $b>0$ is a (not necessarily known) lower bound for $B$. 

\smallskip

It is well known and not difficult to prove that adaptive estimators $\hat f_n = \hat f_n(Y)$ exist that attain the minimax optimal $\ell^2$-risk for every ellipsoid $S^s(B)$:
\begin{equation} \label{adapto}
\sup_{f\in S^{s}(B)} E_{f}\|\hat{f}_n-f\|_2\leq K(s) B^{1/(2s+1)}n^{-s/(2s+1)},~~ \forall s>0,~\forall B >0,
\end{equation}
where the constant $K(s)>0$ depends only on $s$. In fact even exact adaptation to the minimax constant $K(s)$ is possible by suitable Stein-type shrinkage estimators (see Section 3.7 in \cite{TSY}).

In this paper we focus on the construction of \textit{confidence sets} $C_n$ for $f$ in $$\bigcup_{s\in[s_{\min},s_{\max}]} S^{s}(B),~~~B \ge b \text{ arbitrary},$$ that reflect the risk bound (\ref{adapto}) -- that is, we want to find a data-driven subset $C_n$ of $\ell_2$ that contains $f$ with $\Pr_f$-probability at least $1-\alpha$ (where $0<\alpha<1$ is a chosen significance level), and we \textit{also} want $C_n$ to have $\ell_2$-diameter of correct order $B^{1/(2s+1)}n^{-s/(2s+1)}$ up to possibly a multiplicative constant $K'(s)$ (we do not consider adaptation to the exact minimax constant here). Just as $\hat f_n(Y)$ above, $C_n=C_n(Y,\alpha)$ should be adaptive and hence \textit{not} depend on the unknown values $s,B$. 

In the special case $s_{\max} < 2s_{\min}$ this is possible by adapting the proof of Theorem 3A in \cite{BullNickl} to the sequence space setting. However, in the general setting $s_{\max}>2s_{\min}$ relevant in nonparametric statistics, the construction of such a confidence set is not possible, and a valid confidence set always has `worst case' diameter coming from the maximal model $S^{s_{\min}}(B)$ (this follows, e.g., from the proof of Theorem 1 in \cite{BullNickl}, see also Theorem 8.3.5 in \cite{GineNickl_book}). New constraints on the parameter space $S^s(B)$ need to be introduced. For instance, if an upper bound $B_0$ on the radius $B$ known, a testing approach as in \cite{BullNickl} could be used to construct an adaptive confidence set that is honest over a sequence of parameter spaces that asymptotically ($n \to \infty$) contains the \textit{maximal} parameter space $S^{s_{\min}}(B_0)$. It is also proved in Theorem 4 in \cite{BullNickl} that such a result is \textit{impossible without} the bound $B_0$ on $B$ -- for unbounded $B$ some functions from the $s_{\min}$-Sobolev space have to be permanently removed for `honest' inference to be possible (the results in \cite{BullNickl} are in the i.i.d.~sampling model but apply in our simpler setting too). In order to remove `as few functions as possible' we shall consider -- inspired by \cite{PicTri}, \cite{GineNickl}, \cite{Bull} -- a `self-similarity' constraint, which in effect enforces a certain signal-strength condition on the sequence $(f_k: k \in \mathbb N)$.

\subsection{Self-similarity conditions}

For $s \in  [s_{\min}, s_{\max}]$, `self-similarity' function $\epsilon: [s_{\min}, s_{\max}] \to (0,1],~J_0 \in \mathbb N$, $0<b<B<\infty$, and constant $c(s)=16\times2^{2s+1}$, define `self-similar' classes
\begin{align}\label{def: selfsim}
& ~~~~~~~~S^{s}_{\eps(s)} \equiv S^{s}_{\eps(s)}(b,B,J_0)\equiv \\ 
& \bigg\{f \in \ell^2: \|f\|_{s,2} \in [b,B]: \sum_{k {=} 2^{J (1-\epsilon(s))}}^{2^J} f_{k}^2 \ge c(s)\|f\|_{s,2}^2 2^{-2Js}~\forall J \in \mathbb N, J \ge J_0 \bigg\},\notag
\end{align}
where the notation $\sum_{k=a}^{b}c_k$ for $a,b\in\mathbb{R}$ stands for $\sum_{k=\lceil a\rceil}^{\lfloor b\rfloor}c_k$ throughout the whole paper.
Note that $\|f\|_{s,2}<\infty$ implies, for all $J \in \mathbb N$, $$\sum_{k \ge 2^{J (1-\epsilon(s))}} f_k^2 \le \|f\|_{s,2}^2 2^{-2J(1-\epsilon(s))s} = \|f\|_{s,2}^2 2^{-2Js} \times 2^{2J\epsilon(s)s}$$ and for `self-similar' functions this upper bound needs to be matched by a lower bound, accrued repeatedly over coefficient windows $k \in [2^{J (1-\epsilon(s))}, {2^J}], J \ge J_0$, that is not off by more than a factor of $2^{2J\epsilon(s)s}/c(s)$. As a consequence the regularity of $f$ is approximately identified across \textit{all} scales $J \ge J_0$.

If condition $\eqref{def: selfsim}$ holds for some $\eps(s)>0$ then it also holds for $c(s)=16\times 2^{2s+1}$ replaced by an arbitrary small positive constant and any $\eps'(s)>\eps(s)$ (for $J_0$ chosen sufficiently large). In this sense the particular value of $c(s)$ is somewhat arbitrary, and chosen here only for convenience.

Larger values of $\epsilon(s)$ correspond to weaker assumptions on $f$: Indeed, increasing the value of $\epsilon(s)$ makes it easier for a function to satisfy the self-similarity condition, as the lower bound is allowed to accrue over a larger window of `candidate' coefficients, and since the `tolerance factor' $2^{2J\epsilon(s)s}$ in the {lower} bound increases. In contrast, smaller values of $\epsilon(s)$ require a strong enough signal in blocks of comparably small size.

\smallskip

\textit{We shall demonstrate that signal strength conditions enforced through the `self-similarity' function $\epsilon(s)$ allow for the construction of honest adaptive confidence balls over the parameter space$$ \bigcup_{s_{\min} \le s \le s_{\max}} S^s_{\epsilon(s)},$$ with performance resembling the adaptive risk bound (\ref{adapto}). We will effectively show that $$\epsilon(s) <\frac{1}{2}~\forall s$$ is a necessary condition for the construction of such adaptive confidence sets (when $s_{\max}>2s_{\min}$), whereas a sufficient condition is $$\epsilon(s) < \frac{s}{2s+1/2} ~\forall s.$$  As $s \to \infty$ we have $s/(2s+1/2) \to 1/2$, showing that the necessary condition cannot be improved upon.} 
  
\smallskip

Comparing to the self-similarity condition (3.4) in \cite{SzVZ2}, which for $f\in S^{s}(B)$ and transposed into our notation, requires for some $\eta>0$,
\begin{align}
\sum_{k=2^J/\rho}^{2^{J}}f_k^2\geq\eta B^2 2^{-2Js},\quad\text{for all $J\geq J_0$ and some $\rho>1$},\label{def: BayesSelfsim}
\end{align}
one can easily see that the self-similarity condition $\eqref{def: selfsim}$ is strictly weaker, both in terms of the window sizes along which the lower bound has to accrue, and in terms of the lower bound itself. One can show that in the context of \cite{SzVZ2} their stronger assumption is actually necessary (for the particular marginal likelihood empirical Bayes procedure used there). Furthermore we note that (as a consequence of \cite{SzVZ3}) hierarchical Bayes methods behave similarly to the {maximum} marginal likelihood empirical Bayes method in the sense that the self-similarity condition $\eqref{def: BayesSelfsim}$ can not be relaxed. Our results imply that this is an artefact of the above mentioned adaptive Bayesian approaches, and that more refined nonparametric techniques can reach the information-theoretic limits for adaptive confidence sets in $\ell^2$. It is conceivable, however, that an appropriately modified empirical Bayes method might achieve the information theoretic limits derived in the present paper; see \cite{SzVZ3} for some related results and ideas.

\smallskip

Before we proceed with our main results let us clarify that the statistical complexity of the estimation problem did not decrease quantitatively by introducing the self-similarity constraint: The minimax estimation rate over the class $\eqref{def: selfsim}$ is equal to the minimax rate over the Sobolev class $S^{{s}}(B)$. 
\begin{theorem}\label{Thm: minimax}
For any fixed values of $0<b <B, J_0 \in \mathbb N, \eps\in(0,1)$, the minimax rate of estimation over all self-similar functions $S^{s}_\eps \equiv S^{{s}}_{\eps}(b,B,J_0)$ in the Gaussian sequence model $\eqref{model}$ is of order $$\inf_{\hat{T}_n=\hat{T}_n(Y)} \sup_{f \in S^{s}_\eps} E_f \|\hat{T}_n-f\|_2 \simeq n^{-{s}/(2{s}+1)}.$$
\end{theorem}

\subsection{Construction of the confidence ball}\label{sec: construction}

In this subsection we give an algorithm which provides asymptotically honest and adaptive confidence sets over the collection of self-similar functions. As a first step we split the `sample' into two parts ${y'=(y_{k}')}$ and ${y''=(y_{k}'')}$ (with Gaussian noise ${g'_k}$ and ${g''_k}$ {with variance 2}, respectively, see \cite{RobVaart} for instance), inflating the variance of the noise by $2$, with distributions $\Pr_1$ and $\Pr_2$, and expectations $\E_1$ and $\E_2$, respectively. Furthermore {by slightly abusing our notation introduced in the beginning of Section \ref{sec: main} we denote in this section} by $\Pr_f$ and $\E_f$ the joint distribution and the corresponding expected value, respectively. 
 
\smallskip 
 
Using the first sample {$y'$} we denote by $\hat{f}_n(j)$ the linear estimator with `resolution level' (=truncation point) $j\in \mathbb N$,
\begin{equation}\label{eq: linEst}
\hat{f}_n(j) \equiv   {(y_{k}')}_{1\le k \le 2^j}, ~~\E_1\hat{f}_n(j)  =  (f_{k})_{1\le k \le 2^j}= K_j(f),
\end{equation}
where $K_j$ denotes the projection operator onto the first $2^j$ coordinates. Let us consider minimal and maximal truncation levels $j_{\min} = \underline \sigma \log_2 n,~j_{\max} = \overline \sigma \log_2 n$ -- for concreteness we take $\underline \sigma = 1/(2s'+1)$ for arbitrary $s'>s_{\max}$ and $\overline\sigma=1$, but other choices are {also} possible. We define a discrete grid $\mathcal J$ of resolution levels $$\mathcal J = \left\{j \in \mathbb N: j \in [j_{\min}, j_{\max}] \right\} $$ that has approximately $\log_2 n$ elements.
Using Lepski's method define a first estimator by
\begin{equation}\label{def: barj}
{\hat{j}_n} \equiv \min\left\{j \in \mathcal J: \|\hat{f}_n(j)-\hat{f}_n(l)\|_2^2 \le 4\times \frac{2^{l{+1}}}{n}~~\forall l>j, l \in \mathcal J \right\}.
\end{equation}

While ${\hat{j}_n}$ is useful for adaptive estimation via $\hat{f}_n( {\hat{j}_n})$, for adaptive confidence sets we shall need to systematically increase ${\hat{j}_n}$ by a certain amount -- approximately by a factor of two. To achieve this let us take a fixed parameter $0<m<1$ and choose parameters $0<\kappa_1,\kappa_2<1$ that satisfy

\begin{align}
m<\frac{2s_{\min}+1/2}{s_{\min}+(s_{\min}+1/2)/\kappa_1}<1\quad\text{and}\quad 0<\frac{1+\kappa_1}{2\kappa_2}<\kappa_2<1.\label{def: kappa}
\end{align}
Intuitively, given $\delta>0$ we can {choose ${m,}\kappa_1, \kappa_2$ such that} {all} lie in $(1-\delta, 1)$ -- the reader may thus think of {$m$ and} the $\kappa_i$'s as constants that are arbitrarily close to one.  Next an `under-smoothed estimate' is defined as
\begin{align}
\hat{J}_n=\lceil J_n\rceil,\quad\text{where}\quad
\frac{1}{ J_n} \equiv \frac{1}{2\kappa_2}\frac{1}{{\hat{j}_n}}-\frac{1-\kappa_2}{2\kappa_2}\frac{1}{\log_2 n}\label{def: hatJ}.
\end{align}
With $\hat{J}_n$ in hand, we use again the sample {$y'$} to construct any standard adaptive estimator $\hat f_n$ for which the conclusions of Theorem~\ref{thm: AdaptiveEst} in the Appendix hold true, and use the second subsample {$y''$} to estimate the squared $\ell^2$-risk of $\hat f_n$: The risk estimate
$$\tilde{U}_n(\hat f_n) = \sum_{k \le 2^{\hat{J}_n}} ({y''_{k}} - \hat f_{n,k})^2 - \frac{2^{\hat{J}_n{+1}}}{n}$$ has expectation (conditional on the first subsample {$y'$})
\begin{align}
\E_2 \tilde{U}_n(\hat f_n) = \sum_{k \le 2^{\hat{J}_n}} (f_{k}-\hat f_{n,k})^2 = \|K_{\hat{J}_n}(f-\hat f_n)\|_2^2.\label{eq: ERadius}
\end{align}
Our $\ell^2$-confidence ball is defined as
\begin{equation}
C_n = \left\{f: \|f-\hat f_n\|_2^2 \le \tilde{U}_n(\hat f_n) + {\sqrt{8}}\gamma_{\alpha}\frac{2^{\hat{J}_n/2}}{n} \right\},\label{def: Cn}
\end{equation}
where $\gamma_\a$ denotes the $1-\a$ quantile of the standard normal $N(0,1)$ random variable, $0<\alpha<1$.
We note that, unlike \cite{Bull},\cite{BullNickl} or \cite{SzVZ2}, we do not require knowledge of any self-similarity or radius parameters in the construction; we only used the knowledge of $s_{\max}$ in the construction of the discrete grid $\mathcal{J}$ and the parameters $m$ and $s_{\min}$ in the choice of $\kappa_2$.

\smallskip
However, the above construction has also its limitations, it will not work for every self-similarity function $\eps(\cdot)$, hence we have to introduce some additional restriction. Assume that the function $\epsilon(\cdot)$ satisfies
\begin{align}
\sup_{s\in[s_{\min},s_{\max}]}\eps(s)\frac{2s+1/2}{s}\leq m<1,\label{def: eps}
\end{align}
for {a fixed parameter $m\in(0,1)$} introduced in $\eqref{def: kappa}$.

\smallskip

To formulate our main results let us introduce the notation 
\begin{align}
\mathbb S(\eps) = \mathbb{S}(\eps,b,B,J_0)\equiv \cup_{s\in[s_{\min},s_{\max}]}S^{s}_{\eps(s)}(b,B,J_0).\label{def: ColSelfSim}
\end{align} 
for the collection of self-similar functions with regularity ranging between $[s_{\min},s_{\max}]$ and  function $\eps:\,[s_{\min},s_{\max}]\mapsto(0,1)$.

\smallskip

\begin{theorem}\label{thm: Conf}
For any $0<b < B <\infty, J_0 \in \mathbb N,$ and self-similarity function $\eps$ satisfying $\eqref{def: eps}$, the confidence set $C_n$ defined in $\eqref{def: Cn}$ has exact honest asymptotic coverage $1-\a$ over the collection of self-similar functions $\mathbb{S}(\eps)$, i.e.,
\begin{align*}
\sup_{f\in\mathbb{S}(\eps,b,B,J_0)}\Big|{\Pr}_{f}(f\in C_n)-( 1-\a)\Big|\rightarrow0
\end{align*}
as $n\rightarrow\infty$. Furthermore the $\ell^2$-diameter $|C_n|$ of the confidence set is rate adaptive: For every $s\in[s_{\min},s_{\max}], B > b, J_0 \in \mathbb N,$ and $\delta>0$ there exists $C(s, \delta)>0$ such that
\begin{align*}
\limsup_{n\rightarrow\infty}\sup_{f\in S^s_{\eps(s)}(b,B,J_0)}{\Pr}_{f}(|C_n|\geq C(s,\delta)B^{1/(2s+1)}n^{-s/(2s+1)})\leq\delta.
\end{align*}
\end{theorem}

\subsection{Information theoretic lower bound}\label{sec: LBInf}
The assumption of self-similarity in Theorem \ref{thm: Conf} could be entirely removed when $s_{\max}<2s_{\min}$, by adapting the proof of Theorem 3A in \cite{BullNickl} to the sequence space setting considered here. In the more realistic setting $s_{\max}>2s_{\min}$ this is, however, not the case, as our results below will imply. We shall prove that for general adaptation windows $[s_{\min}, s_{\max}]$, the self-similarity function $\eps(s)>0$ can not exceed $1/2$ for an honest and adaptive confidence set to exist over the class $\mathbb S(\eps)$. This will be deduced from the following general lower bound on the size of honest confidence sets for constant self-similarity {function} $\eps(\cdot) \equiv \eps>0$ and two regularity levels $s>r$.

\begin{theorem}\label{Thm: InfLB}
Fix $\a\in(0,1/2),$ $0<\eps(\cdot)\equiv \eps<1$, $0<r<r'<{r/(1-\eps)<}\infty,$ and let  $s\in(r',r/(1-\eps))$ be arbitrary. Then there does not exist a confidence set $C_n$ in $\ell^2$ which satisfies for every $0<b < B, J_0 \in \mathbb N$,
\begin{align}
&\liminf_{n\rightarrow\infty}\inf_{f\in S_{\eps}^{r}(b,B,J_0)\cup S_{\eps}^{s}(b,B,J_0)} {\Pr}_f(f\in C_n)\ge 1-\alpha, \label{eq: Hon}\\
&\sup_{f\in S^{s}_{\eps}(b,B,J_0)}{\Pr}_{f}(|C_n|>r_n) \stackrel{{n \to \infty}}{\to} 0 \label{eq: Adap},
\end{align}
for any sequence $r_n=o(n^{-\frac{r'}{2r'+1/2}})$.
\end{theorem}

\begin{remark}
Theorem \ref{Thm: InfLB} also holds with $r_n=O(n^{-\frac{r'}{2r'+1/2}})$, but for clarity of the proof we decided to state it in the present form. The $r_n=o(n^{-\frac{r'}{2r'+1/2}})$ version of the theorem is already sufficient to prove the next corollary.
\end{remark}

\begin{corollary}\label{Cor: NoAdap}
Assume that $s_{\max}>2s_{\min}$ and $\eps(\cdot) \equiv \eps>1/2$. Then there does not exist a confidence set $C_n$ in $\ell^2$ which satisfies for every $0<b < B, J_0 \in \mathbb N,$
\begin{align}
\liminf_{n\rightarrow\infty}\inf_{f\in \cup_{s\in[s_{\min},s_{\max}]}S_{\eps}^{s}(b,B,J_0)}{\Pr}_f(f\in C_n)\geq1-\alpha,\label{eq: cor1}
\end{align}
and for all $s\in[s_{\min},s_{\max}]$, $\delta>0$, and {some large enough} any $K>0$
\begin{align}
\limsup_{n\rightarrow\infty}\sup_{f\in S^{s}_{\eps}(b,B,J_0)}{\Pr}_{f}(|C_n|>K n^{-s/(2s+1)})\leq\delta.\label{eq: cor2}
\end{align}
\end{corollary}
\begin{proof}
Assume that there exists a honest confidence set $C_n$ satisfying $\eqref{eq: cor1}$ and $\eqref{eq: cor2}$.
Then take any $s\in (2s_{\min},s_{\max})$ and choose the parameters $r,r'$ such that they satisfy $s/2>r'>r>\max\{(1-\eps)s,s_{\min}\}$. Following from Theorem~\ref{Thm: InfLB} if  assertion $\eqref{eq: cor1}$ holds then $\eqref{eq: Adap}$ can not be true, i.e., the size of the confidence set for any $f\in S^{s}_{\eps}(b,B,J_0)$ can not be of a smaller order than $n^{-r'/(2r'+1/2)}$. However, since $r'<s/2$ we have {$n^{-s/(2s+1)}=o( n^{-r'/(2r'+1/2)})$}. Hence the size of the honest confidence set has to be of a polynomially larger order than $n^{-s/(2s+1)}$, which contradicts $\eqref{eq: cor2}$.
\end{proof}

\begin{remark} \normalfont
In Theorem~\ref{thm: Conf} we have proved that for $\eps(s)\leq ms/(2s+1/2)$ (with $s\in [s_{\min},s_{\max}]$ and $m$ arbitrary close to $1$) the construction {of adaptive} and honest confidence sets is possible. The upper bound tends to $1/2$ as $s$ goes to infinity and $m$ to one, showing that the restriction $\eps > 1/2$ in Corollary \ref{Cor: NoAdap} cannot be weakened in general.
\end{remark}

\section{Proof of Theorem~\ref{thm: Conf}}\label{sec: ProofConf}
As a first step in the proof we investigate the estimator of the optimal resolution level ${\hat{j}_n}$ balancing out the bias and variance term{s} in the estimation.
The linear estimator $\hat{f}_{n}(j)$ defined in $\eqref{eq: linEst}$ has bias and variance
so that\begin{equation} \label{bias}
\|\E_1\hat{f}_n(j)-f\|_2^2 \le \|f\|^2_{s,2} 2^{-2js} \equiv B(j,f)
\end{equation}
 and
\begin{equation}
\E_1\|\hat{f}_n(j)-\E_1\hat{f}_n(j)\|_2^2 = (1/n) \E_1 \sum_{k =1 }^{2^j} {g_{k}'}^2 = \frac{2^{j{+1}}}{n}.
\end{equation}

Our goal is to find an estimator which balances out these two terms. For this we used Lepski's method in $\eqref{def: barj}$. For $f \in S^s(B)$ we define
\begin{equation}\label{def: j*}
j_n^* = j_n^*(f)\equiv \min \{j \in \mathcal J: B(j,f) \le 2^{j{+1}}/n \}
\end{equation}
which implies, by monotonicity, that
\begin{equation}\label{eq: help03}
B(j, f) = 2^{-2js} \|f\|_{s,2}^2 \le \frac{2^{j{+1}}}{n}, ~~\forall j \ge j_n^*,~ j\in\mathcal{J},
\end{equation}
\begin{equation*}
B(j,f) = 2^{-2js} \|f\|^2_{s,2} > \frac{2^{j{+1}}}{n},~~\forall j < j_n^*,~ j\in\mathcal{J}.
\end{equation*}
We note that for $n$ large enough (depending only on $b$ and $B$) the inequalities $j_n^*<\lfloor\log_{{2}} n\rfloor$ and $j_n^*>\lceil (\log_{{2}} n)/(2s'+1)\rceil$ hold {(recall that $s'>s_{\max}$ is an arbitrary parameter defined below display $\eqref{eq: linEst}$)}, hence we also have
\begin{align}
2^{2s+1}2^{-2j_n^*s}\|f\|_{s,2}^2\geq \frac{2^{j_n^*{+1}}}{n}.\label{eq: help02}
\end{align}
Therefore we can represent $j_n^*$ and the given value of $s$ as
\begin{align}
j_n^*=\frac{\log_2 n+2(\log_2 (\|f\|_{s,2})+c_n)}{2s+1},~~\text{and}\label{eq: jn*}
\end{align}
\begin{equation} \label{sid}
s = \frac{\log_2 n}{2j_n^*} + \frac{\log_2 (\|f\|_{s,2})+c_n}{j_n^*} - \frac{1}{2},
\end{equation}
respectively, where {$c_n\in[-1/2,s]\subset [-1/2,s_{\max}]$}.

\medskip

The next lemma shows that ${\hat{j}_n}$ is a good estimator for the optimal resolution level $j^*_n$ in the sense that with probability approaching one it lies between $(1-\eps(s))j_n^*$ and $j_n^*$ whenever $f$ is a self-similar function in the sense of (\ref{def: selfsim}).

\begin{lemma}\label{lem: EstRes}
 Assume that $f \in S^s(B)$ for some $s \in [s_{\min}, s_{\max}]$ and any $B>0$.\\
a) We have for all $n \in \mathbb N$, $${\Pr}_{{1}}({\hat{j}_n} \geq j_n^*) \le C\exp\{- 2^{j_n^*}/8\},$$
with $C=2/(1-e^{-1/{8}})^2$.\\
b) Furthermore, if the self-similarity condition $\eqref{def: selfsim}$ holds we also have for all $n \in \mathbb N$ such that $j_n^*\geq J_0$ that
$${\Pr}_{{1}}\big({\hat{j}_n} < j_n^*(1-\eps(s))\big) \leq j_n^*\exp\{-(9/{8})2^{j_n^*}\}.$$
\end{lemma}
\begin{proof}
See Section~\ref{sec: EstRes}.
\end{proof}
We note that by definition $j_n^*\geq\log_2 n/(2s'+1)\rightarrow\infty$ hence for $n$ large enough $j_n^*\geq J_0$ holds uniformly over $f\in\mathbb{S}(\eps,b,B,J_0)$.

As a next step we examine the new (under-smoothed) estimator of the resolution level $\hat{J}_n$. Assuming $f\in S^{s}_{\eps(s)}(b,B,J_0)$, the estimate ${\hat{j}_n}$ of $j_n^*$ can be converted into an estimate of $s$. We note that a given $f$ does not necessarily belong to a unique self-similar class $S^{s}_{\eps(s)}(b,B,J_0)$, but the following results hold for any class $f$ belongs to. We estimate $s$ simply by
$$\bar s_n = \frac{\log_2 n}{2 {\hat{j}_n}}  -\frac{1}{2},$$ ignoring `lower order' terms in (\ref{sid}). We then have from (\ref{sid}) that
\begin{align*}
\bar s_n -s &= \frac{\log_2 n}{2 {\hat{j}_n}}  - \frac{1}{2} - \frac{\log_2 n}{2j^*_n} - \frac{\log_2 (\|f\|_{s,2})+c_n}{j^*_n} +\frac{1}{2}\\
 &= \frac{\log_2 n}{2}\left(\frac{j^*_n-{\hat{j}_n}}{j^*_n {\hat{j}_n}}\right) - \frac{\log_2 (\|f\|_{s,2})+c_n}{j^*_n}.
\end{align*}
Now choose a constant $\kappa_3 \in (\kappa_2, 1)$ so that
$$0<\frac{1+\kappa_1}{2\kappa_2}<\kappa_2<\kappa_3<1,$$
recalling $\eqref{def: kappa}$. From Lemma 1a) we have ${\Pr}_{1}({\hat{j}_n} - j^*_n < 0) \to 1$ uniformly over $f\in\cup_{s\in[s_{\min},s_{\max}]}S^s(B)$, hence from {the inequality $j_n^*\geq(\log_2 n)/(2s'+1)$} we have for some constant $C=C(B, s'), B \ge \|f\|_{s,2}$,
\begin{align*}
{\Pr}_{{1}}\left(\bar s_n \le \kappa_3 s\right) &= {\Pr}_{{1}} \left(\bar s_n - s \le (\kappa_3-1)s\right) \\
&{\leq} {\Pr}_{{1}} \left(\frac{\log_2 n}{2}\left(\frac{{\hat{j}_n}-j^*_n}{j^*_n {\hat{j}_n}}\right) + \frac{\log_2 (\|f\|_{s,2})+c_n}{j^*_n} \ge (1-\kappa_3)s_{\min} \right) \\
& \le {\Pr}_{{1}}\left(C /\log_2 n > (1-\kappa_3)s_{\min} \right) + o(1) {=o(1).}
\end{align*}
On the other hand we also have from Lemma 1b), $\eqref{eq: jn*}$, and $0<\eps(s)\leq1$ that
\begin{align*}
&{\Pr}_{{1}}(\bar s_n \ge (1+\kappa_1)s) \\
&= {\Pr}_{{1}}(\bar s_n- s \ge \kappa_1 s ) \\
&\le {\Pr}_{{1}} \left( \frac{\log_2 n}{2}\left(\frac{j^*_n-{\hat{j}_n}}{j^*_n {\hat{j}_n}}\right) - \frac{{\log_2 (\|f\|_{s,2})+c_n}}{j^*_n} \ge \kappa_1 s  \right) \\
& \le {\Pr}_{{1}}  \left( \frac{j^*_n-{\hat{j}_n}}{j^*_n {\hat{j}_n}} \ge \frac{2\kappa_1 s}{\log_2 n}+\frac{{2(\log_2 (\|f\|_{s,2})+c_n)}}{j_n^*\log_2 n}  \right) \\
& \le {\Pr}_{{1}} \left(\eps(s) j^*_n >  \frac{2\kappa_1 s(1-\eps(s))(j_n^*)^2}{\log_2 n}+\frac{2(1-\eps(s))j_n^*(\log_2 (\|f\|_{s,2})+c_n)}{\log_2 n}  \right)+o(1)\\
&= {\Pr}_{{1}} \Big(\eps(s) >  \frac{2\kappa_1 s(1-\eps(s))}{2s+1}
+\frac{2\kappa_1 s(1-\eps(s))}{2s+1}\times\\
&\qquad\frac{2(\log_2(\|f\|_{s,2})+c_n)}{\log_2 n}+\frac{{2(1-\eps(s))(\log_2 (\|f\|_{s,2})+c_n)}}{\log_2 n}  \Big){+o(1)}\\
&\le {\Pr}_{{1}} \left(\eps(s) >  \frac{2\kappa_1 s(1-\eps(s))}{2s+1}+{\frac{2(\kappa_1+1)s+1}{2s+1}\times}\frac{2\log_2(b{/2})\wedge0}{\log_2 n}\right)+o(1)\\
&= {\Pr}_{{1}} \left(\eps(s) >  \frac{\kappa_1 s}{(1+\kappa_1)s+1/2}+\frac{2\log_2(b{/2})\wedge 0}{\log_2 n}\right)+o(1).
\end{align*}
The probability on the right hand side tends to zero for $n$ large enough (depending only on $b$), since
\begin{align*}
\eps(s)\leq m \frac{s}{2s+1/2}< \frac{\kappa_1(2s_{\min}+1/2)}{(1+\kappa_1)s_{\min}+1/2}\times\frac{s}{2s+1/2}\leq \frac{\kappa_1 s}{(1+\kappa_1)s+1/2}
\end{align*}
following from the definition of $\kappa_1$ given in $\eqref{def: kappa}$ and the monotone increasing property of the function $g(s)=(2s+1/2)/[(1+\kappa_1)s+1/2]$. Therefore we see that on an event of probability approaching one we have
\begin{align}
\bar s_n \in \left(\kappa_3 s, (1+\kappa_1)s\right),\label{eq: SmoothEst}
\end{align}
 and hence if we define $$\hat s_n = \bar s_n/(2\kappa_2)$$ we see
\begin{align}
{\Pr}_{1}\left(\hat s_n \in \left(\frac{\kappa_3}{2\kappa_2}s,\frac{1+\kappa_1}{2\kappa_2}s\right)\right) \to 1\label{eq: EstSmooth}
\end{align}
 as $n \to \infty$.  By choice of the $\kappa_i$'s we see that $\hat s_n$ systematically ``underestimates'' the smoothness $s$ and is contained in a closed subinterval of $(s/2,s)$ with probability approaching one. 
The `resolution level' $J$ corresponding to $\hat s_n$ is $\hat{J}_n$:  Easy (but somewhat cumbersome) algebraic manipulations imply
 \begin{align}
2n^{1/(2\hat{s}_n+1/2)}> 2^{\hat{J}_n}\geq n^{1/(2\hat{s}_n+1/2)}\label{def: hatJ2}
 \end{align}
(where $\hat{J}_n$ was defined in $\eqref{def: hatJ}$). Furthermore we note that from $\eqref{def: hatJ}$ {and $\hat{j}_n\in\mathcal{J}$} also follows
\begin{align}
\hat{J_n}\in\left [\frac{2\kappa_2}{2s'+\kappa_2}\log_2 n, {\lceil 2\log_2 n\rceil} \right].\label{eq: hatJ}
\end{align}
Next we turn our attention to the analysis of the confidence set $C_n$ given in $\eqref{def: Cn}$.
First of all note that
\begin{align}
\tilde{U}_n(\hat f_n)-\E_2 \tilde{U}_n(\hat f_n)&=\frac{1}{n}\sum_{k\leq 2^{\hat{J}_n}}{\big((g_k'')^2-2\big)}+\frac{2}{\sqrt{n}}\sum_{k\leq 2^{\hat{J}_n}}(f_k-\hat{f}_{n,k}){g_k''}\nonumber\\
&\equiv -A_n-{A'}_n. \label{eq: RiskEst}
\end{align}
We deal with the two random sums $A_n$ and ${A'}_n$ on the right hand side separately. First we show that ${A'}_n=O_{{\Pr}_f}(n^{-\frac{2s+1/2}{2s+1}})$. Note that conditionally on the first sample the random variable ${A'}_n$ has Gaussian distribution with mean zero and variance ${(8/n)}\sum_{k\leq \hat{J}_n}(f_k-\hat{f}_{n,k})^2\leq {(8/n)}\|f-\hat{f}_n\|_2^2$ . Furthermore note that $\|f-\hat{f}_n\|_2^2={O_{\Pr_1}}(n^{-\frac{2s}{2s+1}})$ following from the adaptive construction of the estimator $\hat{f}_n$. Hence we can conclude following from the independence of the samples {$y'$ and $y''$}, {and Chebyshev's inequality} that for every $\delta>0$ there exists a large enough constant $K$ such that ${A'}_n\geq Kn^{-\frac{2s+1/2}{2s+1}}$ with ${\Pr}_f$-probability less than $\delta$.

It remains to deal with $A_n$. In view of sample splitting the centered variables $({2-(g_k'')^2})$ are independent of $\hat{J}_n$, have variance $\sigma^2=8$ and finite skewness $\rho>0$. From the law of total probability, $\eqref{eq: SmoothEst}$, $\eqref{eq: hatJ}$ and Berry-Esseen's theorem (Theorem (4.9) in \cite{Durrett}) we deduce that
\begin{align}
\Big|{\Pr}_f &\Big(A_n\leq \frac{\sigma\gamma_{\a}2^{\hat{J}_n/2}}{n}\Big)-(1-\a)\Big|=\nonumber\\
&\leq\sum_{j=2\kappa_2\log_2 n/(2s'+\kappa_2)}^{{\lceil 2\log_2 n\rceil}}\Big|{\Pr}_{{2}} \Big(\frac{1}{\sigma2^{j/2}}\sum_{k=1}^{2^j}{(2-(g_k'')^2)}\leq \gamma_{\a}\Big)-(1-\a)\Big|{\Pr}_{1}(\hat{J}_n=j)\nonumber\\
&\leq {(3\rho/\sigma^3)2^{-\kappa_2\log_2 n/(2s'+\kappa_2)}}=o(1).\label{eq: quantile}
\end{align}
Next note that in view of $f \in S^s(B)$ and Theorem \ref{thm: AdaptiveEst} (using that $\|\hat{f}_n\|_{s,2}$ is uniformly bounded for $f\in S^s(B)$) the bias satisfies that
\begin{align}
{2^{2\hat{J}_n\hat{s}_n}\|K_{{\hat{J}_n}}(f-\hat f_n)-(f-\hat f_n)\|_2^2 =O\Big(2^{2\hat{J}_n\hat{s}_n-2\hat{J}_ns}(\|f\|_{s,2}^2+\|\hat f_n\|_{s,2}^2 )\Big) = o_{\Pr_1}(1)},\label{eq: UBbias}
\end{align}
since $s >[(\kappa_1+1)/(2\kappa_2)]s> \hat s_n$ {with $\Pr_1$-probability tending to 1}. Furthermore following from $\eqref{def: hatJ2}$ we have {$2^{2\hat{J}_n \hat s_n}\geq n2^{-\hat{J}_n/2}$}.
Then by using Pythagoras' theorem, $\eqref{eq: UBbias}$ and  $\eqref{eq: RiskEst}$ we deduce
\begin{align}
{n2^{-\hat{J}_n/2}}\|f-\hat{f}_n\|_2^2&={n2^{-\hat{J}_n/2}}\Big(\|K_{\hat{J}_n}(f-\hat{f}_n)\|_2^2+\|K_{\hat{J}_n}(f-\hat f_n)-(f-\hat f_n)\|_2^2\Big)\nonumber\\
&= {n2^{-\hat{J}_n/2}}\E_2 \tilde{U}_n(\hat f_n)+ {o_{\Pr_1}(1)}\nonumber\\
&= {n2^{-\hat{J}_n/2}}\Big(\tilde{U}_n(\hat f_n)+A_n+{A'}_n\Big)+{o_{\Pr_1}(1)}.\label{eq: help4}
\end{align}
Following from $\eqref{eq: EstSmooth}$ and $\eqref{def: hatJ2}$ we obtain that (uniformly over $\mathbb{S}(\eps,b,B,J_0)$) with $\Pr_{1}$-probability tending to one
$$2^{\hat{J}_n/2}/n\gtrsim n^{-\frac{s(1+\kappa_1)/\kappa_2 }{s(1+\kappa_1)/(\kappa_2)+1/2}},$$
where the right hand side is of larger order than $n^{-\frac{2s}{2s+1/2}}$ by the definition of $\kappa_1$ and $\kappa_2$. Furthermore following from ${A'}_n=O_{\Pr_f}(n^{-\frac{2s+1/2}{2s+1}})$ and $n^{-\frac{2s+1/2}{2s+1}}=o( n^{-\frac{2s}{2s+1/2}})$ we see that the right hand side of $\eqref{eq: help4}$ can be rewritten as
\begin{align}
{n2^{-\hat{J}_n/2}}\tilde{U}_n(\hat f_n)+A_n+ {o_{{\Pr}_f}(1)}\label{eq: help5}.
\end{align}
Therefore following from $\eqref{eq: help4}$, $\eqref{eq: help5}$ and $\eqref{eq: quantile}$ we deduce that the confidence set $C_n$ given in $\eqref{def: Cn}$ has exact asymptotic coverage $1-\a$
\begin{align*}
{\Pr}_f(f\in C_n)&={\Pr}_f\Big({n2^{-\hat{J}_n/2}}\|f-\hat{f}_n\|_2^2\leq {n2^{-\hat{J}_n/2}}\tilde{U}_n(\hat f_n)+{\sqrt{8}\gamma_\alpha}\Big)\\
&={\Pr}_f\Big({n2^{-\hat{J}_n/2}}A_n\leq {\sqrt{8}\gamma_\alpha+o_{{\Pr}_f}(1)}\Big)= 1-\alpha+o(1).
\end{align*}

Finally we show that the radius of the confidence set is rate adaptive.
First we note that
$$ 2^{\hat{J}_n/4}/\sqrt n \leq 2^{1/4} n^{-\hat s_n/(2\hat s_n+1/2)} = o_{{{\Pr}_1}}(n^{-s/(2s+1)}),$$
following from ${s>}\hat s_n >s\kappa_3/(2\kappa_2)> s/2$ {with $\Pr_1$-probability tending to 1}  and $\eqref{def: hatJ2}$.
Then following from $\eqref{eq: ERadius}$ and Theorem~\ref{thm: AdaptiveEst} we conclude
$$\E_{f}\tilde{U}_n(\hat{f}_n)=\E_1\|K_{\hat{J_n}}{(f-\hat{f}_n)}\|_2^2\leq \E_1\|f-\hat{f}_n\|_2^2\leq K(s)B^{1/(1+2s)}n^{-s/(1+2s)},$$ so that the second claim of Theorem \ref{thm: Conf} follows from Markov's inequality.

\subsection{Proof of Lemma~\ref{lem: EstRes}}\label{sec: EstRes}
a) Pick any $j \in \mathcal{J}$ so that $j> j_n^*$ and denote by $j^-=j-1 \ge j_n^*$ the previous element in the grid. One has, by definition of ${\hat{j}_n}$,
\begin{equation} \label{pr}
{\Pr}_{1} ({\hat{j}_n}= j)  \leq  \sum_{l \in \mathcal{J}: l \ge j} {\Pr}_{1} \left ( \left\|\hat{f}_n(j^-) - \hat{f}_n(l) \right\|^2_2 > 4\times \frac{2^{l{+1}}}{n}\right),
\end{equation}
and we observe that
\begin{equation*}
\left\|\hat{f}_n(j^-) - \hat{f}_n(l) \right \|^2_2 = \frac{1}{n}\sum_{k=2^{j^-}+1}^{2^l}{g_k'^2}+\sum_{k=2^{j^-}+1}^{2^l}f_k^2-\frac{2}{\sqrt{n}}\sum_{k=2^{j^-}+1}^{2^l}f_k {g_k'}.
\end{equation*}
Since $f\in S^{s}(B)$ and $l\geq j^-\geq j_n^{*}$ we have
\begin{align}
\sum_{k=2^{j^-}+1}^{2^l}f_k^2\leq \|f\|_{s,2}^2 2^{-2sj^{-}}=B(j^{-},f)\leq \frac{2^{j^-{+1}}}{n}\leq\frac{2^{l{+1}}}{n}.\label{eq: help01}
\end{align}
Therefore each probability in $\eqref{pr}$ are bounded from above by the sum of the following probabilities
\begin{align}
{\Pr}_{1}\Big(\frac{1}{n}\sum_{k=2^{j^-}+1}^{2^l}{g_{k}'}^2\geq 2\times\frac{2^{l{+1}}}{n}\Big)\leq {\Pr}_{1}\Big(\sum_{k=2^{j^-}+1}^{2^l}({g_{k}'}^2-2)\geq 2^{l{+1}}\Big)\label{eq: probChi}
\end{align}
and
\begin{align}
{\Pr}_{{1}}\Big(\Big|\frac{2}{\sqrt{n}}\sum_{k=2^{j^-}+1}^{2^l}f_k {g_k'}\Big|\geq \frac{2^{l{+1}}}{n}  \Big)\leq {\Pr}_{{1}}\Big(|Z|\geq \frac{2^l}{\sqrt{n}}\Big),\label{eq: probNorm}
\end{align}
where $Z$ is a Gaussian distributed random variable with mean zero and variance {$2{\sum_{k=2^{j^-}+1}^{2^l}} f_k^2\leq 2^{l+2}/n$} following from $\eqref{eq: help01}$.
Then by Theorem \ref{thm: help1} (with $t=2^{l{+1}}$, ${\sigma^2=2}$, and $n=2^l-2^{j^{-}}$) the right hand side of $\eqref{eq: probChi}$ is bounded from above by
$$\exp\Big\{-\frac{2^{{2l+2}}/4}{4(2^{l+1}/2+2^l-2^{j^{-}})}\Big\}\leq \exp\Big\{-\frac{2^{{2l}}}{4(2^l+2^l)}\Big\}\leq e^{-2^l/{8}}.$$
Furthermore by a standard Gaussian tail bound the probability in $\eqref{eq: probNorm}$ is bounded by
$$ \frac{2}{\sqrt{2\pi}2^{l/2-1}}\exp\{ -2^{l-2}\}\leq e^{-2^l/{4}}.$$
We thus obtain that
\begin{align*}
{\Pr}_{1} ({\hat{j}_n}= j)\le \sum_{l=j}^{\log_2 n}2e^{-2^l/{8}}\leq \frac{2}{1-e^{-1/{8}}}e^{-2^j/{8}},\,j\geq j_n^*,\\
{\Pr}_{1} ({\hat{j}_n}\geq j^*_n)\leq \sum_{j=j_n^*}^{\log_2 n}{\Pr}_{1} ({\hat{j}_n}= j)\leq \frac{2}{(1-e^{-1/{8}})^2}e^{-2^{j_n^*}/{8}}.
\end{align*}

\medskip

{For Part b), fix $j \in \mathcal J$ such that $j<j_n^*(1-\eps)$, where $\eps=\eps(s)$. Then by definition of ${\hat{j}_n}$}
\begin{align}
{\Pr}_{1} ( {\hat{j}_n} =j) \le {\Pr}_{{1}}(\|\hat{f}_n(j)-\hat{f}_n(j_n^*)\|_2 \le 2 \sqrt{2^{j_n^*{+1}}/n}).\label{eq: help04}
\end{align}
Now, using the triangle inequality
\begin{align*}
\|\hat{f}_n(j)-\hat{f}_n(j_n^*)\|_2 & = \|\hat{f}_n(j)-\hat{f}_n(j_n^*)-\E_1(\hat{f}_n(j)-\hat{f}_n(j_n^*)) + \E_1(\hat{f}_n(j)-\hat{f}_n(j_n^*))\|_2 \\
& \ge \|\E_1(\hat{f}_n(j)-\hat{f}_n(j_n^*))\|_2 - \|\hat{f}_n(j)-\hat{f}_n(j_n^*)-\E_1(\hat{f}_n(j)-\hat{f}_n(j_n^*))\|_2 \\
& = \sqrt{\sum_{k=2^j+1}^{2^{j_n^*}} f_k^2} - \frac{1}{\sqrt n}\sqrt{\sum_{k=2^j+1}^{2^{j_n^*}} {g_k'}^2}.
\end{align*}
Since $j<j_n^*(1-\eps)$ we have from the definition of self-similarity $\eqref{def: selfsim}$ and $\eqref{eq: help02}$ that
$$\sqrt{\sum_{k=2^j+1}^{2^{j_n^*}} f_k^2} \ge \sqrt{\sum_{k=2^{j_n^*(1-\eps)}}^{2^{j_n^*}} f_k^2} \ge 4 \times 2^{s+1/2} \|f\|_{s,2} 2^{-j_n^*s} \ge 4 \times \sqrt{\frac{2^{j_n^*{+1}}}{n}}, $$
so that the probability on the right hand side of $\eqref{eq: help04}$ is less than or equal to
\begin{align*}
& {\Pr}_{1} \left(\frac{1}{\sqrt n} \sqrt{\sum_{k=2^{j}+1}^{2^{j^*_n}} {g_{k}'}^2}  \ge \sqrt{\sum_{k=2^j+1}^{2^{j_n^*}} f_{k}^2} - 2\sqrt{\frac{2^{j_n^*{+1}}}{n}}    \right)  \\
& \le  {\Pr}_{1} \left(\sum_{k=1}^{2^{j_n^*}} {g_{k}'}^2  >  (4 -2)^2 2^{j_n^*{+1}}    \right) \\
& =  {\Pr}_{1} \left(\sum_{k=1}^{2^{j_n^*}} ({g_{k}'^2-2})  >  3\times 2^{j_n^*{+1}} \right).
\end{align*}
This probability on the right hand side is bounded by $\exp\{-(9/{8})2^{j_n^*}\}$ following from Theorem \ref{thm: help1} (with $t=3\times 2^{j_n^*{+1}}, \sigma^2=2$ and $n=2^{j_n^*}$). The overall result follows by summing the above bound in $j<(1-\eps)j_n^*<j_n^*$.
\bigskip

\section{Proof of Theorem \ref{Thm: InfLB}}\label{sec: ProofInfLB}
The proof of the theorem adapts ideas from the proof of Theorem 4 of \cite{BullNickl}. {In this section we use the notation ${\Pr}_f^{(n)}$ and $\E_f^{(n)}$ introduced in Section \ref{sec: main} for the distribution and expected value of $y$, defined in  $\eqref{model}$, respectively (there is no sample splitting in this case as in Section \ref{sec: construction}). As a special case we note that ${\Pr}_0^{(n)}$ and $\E_0^{(n)}$ denotes the distribution and expected value of $y_k=g_k/\sqrt{n},\,k\in\mathbb{N}$, respectively.}

Let us assume that such a confidence set $C_n$ exists and derive a contradiction with the help of a particularly constructed sequence $(f_m: m \in \mathbb N)$ of $s$-self-similar functions. We denote the limit of these sequences by $f_{\infty}$, which will also be shown to be {$r$-}self-similar. Then we show that along a subsequence $n_m$ of $n$, and for $\delta = (1-2\alpha)/5>0$,
\begin{align}
\sup_{m}\P_{f_{\infty}}^{(n_m)}\big(f_{\infty}\in C_{n_m}\big)\leq 1-\a-\delta\label{eq: Contradiction}
\end{align}
contradicting $\eqref{eq: Hon}$.

\smallskip

We partition $\mathbb{N}$ into sets of the form $Z_{i}^0=\{2^{i},2^{i}+1,...,2^{i}+2^{i-1}-1\}$ and $Z_{i}^1=\{2^{i}+2^{i-1},2^{i}+2^{i-1}+1,...,2^{i+1}-1\}$.
Let us choose a parameter $s'>s$ satisfying $r>s'(1-\eps)>s(1-\eps)$ and define self-similar sequences $f_m=(f_{m,k})$, for $m\in\mathbb{N}$,
$$
f_{m,k}=\begin{cases}
2^{-(s'+1/2)l} & \text{for $l\in\mathbb{N} \cup \{0\}$ and $k\in Z_l^0$,}\\
2^{-(r'+1/2)j_i}\beta_{j_i,k} & \text{for $i\leq m$ and $k\in Z_{j_i}^1$,}\\
0 & \text{else,}
\end{cases}
$$
for some {monotone increasing sequence $j_i\in\mathbb{N}$} tending to infinity and coefficients $\beta_{j_i,k}=\pm 1$ to be defined later. First we show that independently of the choice of the {monotone increasing} sequence $j_i$ and of the coefficients $\beta_{j_i,k}=\pm 1$, the signals $f_{m}$ and $f_{\infty}$ satisfy the self-similarity condition.

Using the definition of $f_m$, the monotone decreasing property of the function $f(x)=x^{-1-2(s'-s)}$and the inequality $s>r'$ one can see that
\begin{align}
\|f_m\|_{s,2}^2&=\sum_{k=1}^{\infty}f_{m,k}^2k^{2s} \le 2^{2s'+1}\sum_{k=1}^{\infty}k^{-1-2(s'-s)}+ 2^{2s}\sum_{i=1}^{m}\sum_{k\in Z_{j_i}^1} 2^{j_i(2s-2r'-1)}\nonumber\\
&\leq 2^{2s'+1}(1+\int_1^{\infty}x^{-1-2(s'-s)}dx)+2^{2s-1}\sum_{i=1}^{m}2^{j_i(2s-2r')}\nonumber\\
&\leq 2^{2s'+1}(1+\frac{1}{2(s'-s)})+2^{2s-1}\frac{2^{j_m(2s-2r')}}{1-2^{-(2s-2r')}}\equiv B(s,s',r',j_m),\label{eq: help2}
\end{align}
{where the} constant $B(s,s',r',j_m)$ {depends} only on $s,s',r'$ and $j_m$. Furthermore
\begin{align}
\sum_{{k=2^{(1-\eps)J}}}^{2^J}f_{m,k}^2\geq \sum_{k\in Z_{{\lceil(1-\eps)J\rceil}}^0}f_{m,k}^2= 2^{-(2s'+1){\lceil(1-\eps)J\rceil}}\times 2^{{\lceil(1-\eps)J\rceil}-1}=  2^{-2s'{\lceil(1-\eps)J\rceil}}/2,\label{eq: selfsimf_m}
\end{align}
{for  $J\geq\lceil (1-\eps)J\rceil+1$, which holds for $J\geq J_0$ (where $J_0$ depends only on $\eps$).}
Then following from the upper bound on the norm $\eqref{eq: help2}$ and the inequalities $s'(1-\eps)<r<s$ the right hand side of $\eqref{eq: selfsimf_m}$ is further bounded from below by
$$2^{-2{r}J }/2\geq 16\times2^{2s+1}B(s,s',r',j_m) 2^{-2s J}\geq 16\times2^{2s+1}\|f_m\|_{s,2}^2 2^{-2s J},$$
for $J>J_0$ (where $J_0$ depends on $s,s',r,{r',}\eps$ and $j_m$). [{We} note that the dependence of $J_0$ on $j_m$ is harmless since {$n_m$ is defined independently of $j_m$, see below.} ] Finally the lower bound on the Sobolev norm can be obtained via
\begin{align}
\|f_m\|_{s,2}^2\geq \sum_{k\in Z_1^0}f_{m,k}^2k^{{2s}}=2^{-1-2(s'-s)}>2^{-1-2(s'-r)}\equiv b^2.\label{eq: selfsimLB}
\end{align}

Next we show that $f_\infty$ is $r$-self-similar. First we note that the existence of $f_{\infty}$ follows from the Cauchy property of the sequence $(f_m)$ in $\ell^2$. Furthermore by definition we have that $f_{\infty,k}=f_{m,k}$ for all $k\leq 2^{j_m}, m \in \mathbb N$. Therefore similarly to $\eqref{eq: selfsimLB}$ and $\eqref{eq: help2}$ the signal $f_{\infty}$ satisfies $\|f_{\infty}\|_{r,2}\geq b$ and
\begin{align}
\|f_{\infty}\|_{r,2}^2&=\sum_{k=1}^{\infty}f_{\infty,k}^2k^{2r}\leq 2^{2r'+1}\sum_{k=1}^{\infty} k^{-1-2(r'-r)}\nonumber\\
&\leq 2^{2r'+1}(1+\frac{1}{2(r'-r)})\equiv B(r,r'),\label{eq: help3}
\end{align}
hence it belongs to the Sobolev ball $S^{r}(B)$ with radius $B=B(r,r')$ depending only on $r$ and $r'$. Then similarly to $\eqref{eq: selfsimf_m}$ we deduce from $\eqref{eq: help3}$ and the inequality $(1-\eps)s'<r$ that
\begin{align*}
\sum_{k=2^{{(1-\eps)J}}}^{2^J}f_{\infty,k}^2&\geq 2^{-2s'{\lceil(1-\eps)J\rceil}}/2 \geq 16\times2^{2r+1}B(r,r') 2^{-2r J} \\
& \geq 16\times2^{2r+1}\|f_{\infty}\|_{r,2}^2 2^{-2r J},
\end{align*}
for $J>J_0$ (where $J_0$ depends only on {$r,r',s'$ and $\eps$}).

Next we give a recursive algorithm for the choice of the sequence $j_m$ and the parameters $(\beta_{j_i,k}:\, k\in Z_{j_i}^1)$. We start the sequence with $j_0=1$ and $n_0=1$. If we assume that for $0\leq i\leq m-1$ the parameters $j_i$ and $(\beta_{j_i,k}:\, k\in Z_{j_i}^1)$ are already chosen, then for $n_m$ large enough (depending only on {$j_{m-1}$ (through $f_{m-1}$)}, $\delta$ and not on {$j_m$}) we have from $\eqref{eq: Hon}$ and $\eqref{eq: Adap}$ that
\begin{align}
\P_{f_{m-1}}^{(n_m)}(f_{m-1}\notin C_{n_m})\leq\a+\delta,\label{eq: coverage}\\
\P_{f_{m-1}}^{(n_m)}(|C_{n_m}|\geq { r_{n_m}})\leq\delta\label{eq: size},
\end{align}
{with $f_{m-1}\in S^{s}_{\eps}(b,B,J_0)$,} where ${b},J_0$ and $B$ depend only on $s,s',r,r',j_{m-1},\eps$ and are independent of $n_m$. Then we choose $j_m$ such that
\begin{align}
n_m=c2^{j_{m}(2r'+1/2)},\label{def: jm}
\end{align}
with a small enough constant $c$ satisfying 
\begin{align}
e^{c^2/2}\leq 1+{\delta^2}.\label{def: c}
\end{align} 
We note furthermore that $n_m$ has to be chosen large enough such that $j_m/j_{m-1}$ is at least $1+1/(2r')$.

Next we define the coefficients $\{\beta_{j_m,k}:\,k\in Z_{j_m}^1\}$. Let the $k$th coefficient of the sequence $f_{m,\beta}$ be
$$f_{m,\beta,k} = f_{m-1,k} + \beta_{j_m,k}{2^{-(r'+1/2)j_m}} 1_{\{k \in Z_{j_m}^1\}},~~k \in \mathbb N,$$ denote the sequence derived from the sequence $f_{m-1}$ by adding the coefficients $\{\beta_{j_m,k}{2^{-(r'+1/2)j_m}}:\,k\in Z_{j_m}^1\}$.
Then define
$$Z_{\beta}=\frac{d\Pr_{f_{m,\beta}}^{(n_m)}}{d\Pr_{f_{m-1}}^{(n_m)}}$$
and set $Z=2^{-2^{j_m-1}}\sum_{\beta}Z_{\beta}$, {so $\E_{f_{m-1}}^{(n_m)}(Z)=1$}. Let us introduce the notation $$\gamma_{n_m}={n_m}2^{-(2r'+1)j_m}.$$
Then from Proposition \ref{prop: help1} we have that
\begin{align*}
Z&={ 2^{-2^{j_m-1}} \sum_{\beta} \frac{d\Pr_{f_{m,\beta}}^{(n_m)}/d\Pr_{0}^{(n_m)}}{d\Pr_{f_{m-1}}^{(n_m)}/d\Pr_{0}^{(n_m)}}}\\
&= 2^{-2^{j_m-1}}\sum_{\beta}\exp\{{n_m}\sum_{k=1}^{\infty}(f_{m,\beta,k}-f_{m-1,k}){y_{k}} +{n_m}(\|f_{m,\beta}\|_2^2-\|f_{m-1}\|_2^2
)/2 \}\\
&=2^{-2^{j_m-1}}\sum_{\beta}\exp\{\sum_{k\in Z_{j_{m}}^1}({\sqrt{n_m}}\beta_{j_m,k}\sqrt{\gamma_{n_m}}{y_{k}})-2^{j_m-1}(\gamma_{n_m}/2)\}.\\
&=2^{-2^{j_m-1}}\sum_{\beta}\prod_{k\in Z_{j_{m}}^1}\exp\{{\sqrt{n_m}}\beta_{j_m,k}\sqrt{\gamma_{n_m}}{y_{k}}-\gamma_{n_m}/2\}.
\end{align*}
By applying Fubini's theorem, Proposition \ref{prop: help1}, the formula $\E_0^{(n)} e^{\sqrt{n}uy_k}=e^{u^2/2}$ and that $f_{m-1,k}=0$ for $k\in Z_{j_m}^1$ we see that
\begin{align*}
\E_{f_{m-1}}^{(n_m)}Z^2&=\E_{f_{m-1}}^{(n_m)}\Big(2^{-2^{j_m-1}}\sum_{\beta} \prod_{k\in Z_{j_{m}}^1}\exp\{{\sqrt{n_m}}\beta_{j_m,k}\sqrt{\gamma_{n_m}}{y_{k}}-\gamma_{n_m}/2\} \Big)^2\\
&= 2^{-2^{j_m}}\sum_{\beta,\beta'}\E_0^{(n_m)}\Big(\prod_{k\in Z_{j_{m}}^1}\exp\{{\sqrt{n_m}}(\beta_{j_m,k}+\beta'_{j_m,k})\sqrt{\gamma_{n_m}}{y_{k}}-\gamma_{n_m}\}\times\\
&\qquad \exp\{{n_m}\sum_{k=1}^{\infty}f_{m-1,k}{y_{k}}-{n_m}\|f_{m-1}\|_2^2/2\}\Big)\\
&=2^{-2^{j_m}}\sum_{\beta,\beta'} \prod_{k\in Z_{j_{m}}^1}\exp\{\frac{\gamma_{n_m}}{2}(\beta_{j_m,k}+\beta'_{j_m,k})^2-\gamma_{n_m} \}\\
&=2^{-2^{j_m}}\sum_{\beta,\beta'} \exp\{\gamma_{n_m}\sum_{k\in Z_{j_{m}}^1}\beta_{j_m,k}\beta'_{j_m,k}\}\\
&=\E\big(\exp\{\gamma_{n_m}Y_{j_m}\}\big),
\end{align*}
where $Y_{j_m}=\sum_{i=1}^{2^{j_m-1}}R_i$ for i.i.d. Rademacher random variables $R_i$ and $\E$ is the {corresponding expectation.}

Note that {following $\eqref{def: jm}$} {$c=n_m2^{-(2{r'}+1/2)j_m}=\gamma_{n_m} 2^{j_m/2}$} and recall the definition of the hyperbolic cosine function $\cosh(x)=(e^{x}+e^{-x})/2$. Then we deduce that
\begin{align*}
\E\big(\exp\{\gamma_{n_m}Y_{j_m}\}\big)&=\E\big(\exp\{c 2^{-j_m/2}\sum_{i=1}^{2^{j_m-1}}R_i\}\big)\\
&=\Big(\frac{e^{-c 2^{-j_m/2}}+e^{c 2^{-j_m/2}}}{2}\Big)^{2^{j_m-1}}\\
&= \cosh(c 2^{-j_m/2})^{2^{j_m-1}}\\
&=\big(1+c^2 2^{-j_m}(1+o(1))\big)^{2^{j_m-1}}\\
&\leq \exp\big\{ c^2(1/2+o(1))\big\}\\
&\leq 1+\delta^2,
\end{align*}
using the definition of $c$ given in $\eqref{def: c}$. Conclude  that therefore 
 {
\begin{align}
\E_{f_{m-1}}^{(n_m)}(Z-1)^2&=\E_{f_{m-1}}^{(n_m)}(Z-\E_{f_{m-1}}^{(n_m)}Z)^2\nonumber\\
&{=} \E_{f_{m-1}}^{(n_m)}Z^2-(\E_{f_{m-1}}^{(n_m)}Z)^2\leq 1+\delta^2-1
\leq\delta^2. \label{eq: help3_0}
\end{align}}
As a consequence of the preceding inequality if we consider the test $T_{n_m}=1\{\exists f\in C_{n_m}, \|f-f_{m-1}\|_2\geq {r_{n_m}}\}$ then by the Cauchy-Schwarz {and Jensen's} inequality
\begin{align}
&{\Pr}_{f_{m-1}}^{(n_m)}(T_{n_m}=1)+\max_{\beta}\P_{f_{m,\beta}}^{(n_m)}(T_{n_m}=0)\nonumber\\
&\quad\geq{\Pr}_{f_{m-1}}^{(n_m)}(T_{n_m}=1)+2^{-2^{j_m-1}}\sum_{\beta}\P_{f_{m,\beta}}^{(n_m)}(T_{n_m}=0)\nonumber\\
&\quad= 1+ \E_{f_{m-1}}^{(n_m)}[(Z-1) 1\{T_{n_m}=0\}]\nonumber\\
&\quad\geq 1-\delta.\label{eq: help6}
\end{align}
We set $f_m$ equal to $f_{m,\beta}$ maximizing the preceding expression in $\beta$.

Then for the limiting sequence $f_{\infty}$ we can likewise compute the likelihood ratio
$$Z'=\frac{d{\Pr}_{f_{\infty}}^{(n_m)}}{d{\Pr}_{f_{m}}^{(n_m)}}.$$
We have that $\E_{f_m}^{(n_m)} [Z']=1$ and
\begin{align}
\|f_{\infty}-f_m\|_2^2&=\sum_{i=m+1}^{\infty}\sum_{k\in Z_{j_i}^1}2^{-(2r'+1)j_i}\leq (1/2)\sum_{i=m+1}^{\infty}2^{-2r'j_{i}}\nonumber\\
&\leq \frac{2^{-2r' j_{m+1}}}{2-2^{1-2r'}}\leq \frac{2^{-(2r'+1) j_{m}}}{2-2^{1-2r'}} ,\label{eq: help1}
\end{align}
following from the definition of $j_m$. Let us denote by $\gamma_{{n_m},j_i}={n_m}2^{-(2r'+1)j_i}$. Then similarly to the computation of $\E_{f_{m-1}}^{(n_m)}[Z^2]$ we have
\begin{align*}
\E_{f_m}^{(n_m)}[Z'^2]&= \E_{f_m}^{(n_m)}\Big( \prod_{i=m+1}^{\infty}\prod_{k\in Z_{j_i}^1} \exp\{\beta_{j_i,k}{\sqrt{n_m\gamma_{n_m,j_i}}y_{k}}-\gamma_{n_m,j_i}/2 \} \Big)^2\\
&= \E_{0}^{(n_m)}\Big( \prod_{i=m+1}^{\infty}\prod_{k\in Z_{j_i}^1} \exp\{2{\sqrt{n_m}}\beta_{j_i,k}\sqrt{\gamma_{n_m,j_i}}{y_{k}}-\gamma_{{n_m},j_i} \} \times\\
&\quad \exp\{{n_m}\sum_{k=1}^{\infty}f_{m,k}{y_{k}}-{n_m}\|f_{m}\|_2^2/2\}\Big)\\
&=\prod_{i=m+1}^{\infty}\prod_{k\in Z_{j_i}^1} \exp\{2\beta_{j_i,k}^2\gamma_{{n_m},j_i}-\gamma_{{n_m},j_i} \}\\
&=\exp\{n_m\|f_{\infty}-f_m\|_2^2\},
\end{align*}
where the right hand side following from $\eqref{def: jm}$ and $\eqref{eq: help1}$ is bounded from above by
\begin{align*}
\exp\{D 2^{-j_m/2}\}\leq 1+\delta^2,
\end{align*}
for some positive constant $D$ (depending only on $r$ and $c$) and $m$ large enough. Hence similarly to $\eqref{eq: help3_0}$ $\E^{(n_m)}_{f_m}[(Z'-1)^2]\leq\delta^2$ which together with $\eqref{eq: help6}$ leads to
\begin{align}
\P_{f_{m-1}}^{(n_m)}(T_{n_m}=1)+\P_{f_{\infty}}^{(n_m)}(T_{n_m}=0)&=\P_{f_{m-1}}^{(n_m)}(T_{n_m}=1)+\E_{f_m}^{(n_m)}[Z'1\{T_{n_m}=0\}]\nonumber\\
&\geq1-\delta+\E_{f_{m}}^{(n_m)}[(Z'-1)1\{T_{n_m}=0\}]\nonumber\\
&\geq 1-2\delta.\label{eq: help0}
\end{align}

Now if $C_{n_m}$ is a confidence set as in the theorem satisfying $\eqref{eq: coverage}$ and $\eqref{eq: size}$ then we have from the definition of the test $T_{n_m}$ that
\begin{align*}
{\P}_{f_{m-1}}^{(n_m)}(T_{n_m}=1)\leq {\P}_{f_{m-1}}^{(n_m)}(f_{m-1}\notin C_{n_m})+{\P}_{f_{m-1}}^{(n_m)}(|C_{n_m}|\geq {r_{n_m}})\leq \a+2\delta,
\end{align*}
which combined with the previous display gives
\begin{align*}
\P_{f_{\infty}}^{(n_m)}(T_{n_m}=0 )\geq 1-\a-4\delta.
\end{align*}

By construction and $\eqref{def: jm}$ we have
\begin{align*}
\|f_\infty-f_{m-1}\|_2^2\geq \sum_{k\in Z_{j_m}^1} \beta_{k}^2 2^{-(2r'+1)j_m }=2^{-2r'j_m}/2
=(c^{\frac{2r'}{2r'+1/2}}/2)n_m^{-2r'/(2r'+1/2)},
\end{align*}
{and $r_{n_m}^2=o(n_m^{-2r'/(2r'+1/2)})$}
hence the event $f_{\infty}\in C_{n_m}$ implies that $C_{n_m}$ contains an element ($f_{\infty}$) that is at least {$ r_{n_m}$} far away from $f_{m-1}$. We deduce the desired contradiction
 \begin{align*}
\P_{f_\infty}^{(n_m)}(f_{\infty}\in C_{n_m})\leq \P_{f_{\infty}}^{(n_m)}(T_{n_m}=1)\leq \a+4\delta=1-\a-\delta.
\end{align*}

\section{Proof of Theorem~\ref{Thm: minimax}}\label{sec: minimax}
The proof is a standard minimax lower bound after checking that the least favourable `prior' concentrates on self-similar functions. Note that $S^{{s}}_{\eps({s})}(b,B,J_0)$ is a subset of $S^{{s}}(B)$ hence it is sufficient to show that the minimax rate over $S^{{s}}_{\eps(r)}(b,B,J_0)$ is bounded below by a small enough constant multiplier of $n^{-{s}/(1+2{s})}$. For notational simplicity we write $\eps=\eps({s})$.

For fixed $0<b<B<\infty$ and given noise level we construct a set of ${s}$-self-similar functions $\{f_m:\,m\in\mathcal{M}\}$ and a benchmark ${s}$-self-similar function $f_0$. First we show that the signals $f_m$ are sufficiently far away from each other with respect to the $\ell^2$-norm (constant times the minimax rate far away). Then we show that their Kullback-Leibler divergence $K(\cdot,\cdot)$ from $f_0$ is small enough to apply Theorem~\ref{thm: help2} {in the Appendix}.

Take ${r}>{s}$ such that ${s}>(1-\eps){r}$ and using the notations of Theorem \ref{Thm: InfLB} let $Z_{i}^0=\{2^{i},2^{i}+1,...,2^{i}+2^{i-1}-1\}$ and $Z_{i}^1=\{2^{i}+2^{i-1},2^{i}+2^{i-1}+1,...,2^{i+1}-1\}$. Then we define $f_0,f_{m,j}\in\ell^2$ as
$$
f_{0,k}=\begin{cases}
K_12^{-({r}+1/2)l} & \text{for $l\in\mathbb{N}$ and $k\in Z_l^0$,}\\
0 & \text{else,}
\end{cases}
$$
and
$$
f_{m,j,k}=\begin{cases}
K_12^{-({r}+1/2)l} & \text{for $l\in\mathbb{N}$ and $k\in Z_l^0$,}\\
\delta \beta_{m,j,k} 2^{-({s}+1/2)j} & \text{for $k\in Z_{j}^1$,}\\
0 & \text{else,}
\end{cases}
$$
for some coefficients $\beta_{m,j,k}\in\{0,1\}$ and $K_1,\delta>0$ to be defined later. Next we show that all the above defined sequences $f_0$ and $f_{m,j}$ are ${s}$-self-similar.

First of all we show that their $\|\cdot\|_{{s},2}$-norm is bounded from below by $b$. From definition we have
$$\|f_{m,j}\|_{{s},2}^2\geq\|f_{0}\|_{{s},2}^2=K_1^2\sum_{l\in\mathbb{N}}\sum_{k\in Z_l^0}2^{-(1+2{r})l}k^{2{s}},$$
where the right hand side {is finite and} depends only on the choice of $s$ and $r$. We choose $K_1$ such that the right hand side of the preceding display is equal to $b^2$.

 As a next step we verify that $f_0$ and $f_{m,j}$ are in $S^{{s}}(B)$
\begin{align*}
\|f_{0}\|_{{s},2}^2&\leq \|f_{m,j}\|_{{s},2}^2=\sum_{k=1}^{\infty}f_{m,j,k}^2k^{2{s}}\\
&\leq {K_1^2\sum_{l\in\mathbb{N}}}\sum_{k\in Z_l^0}2^{-(1+2{r})l}k^{2{s}}+2^{2{s}}\delta^2\sum_{k\in Z_{j}^1}\beta_{m,j,k}^2 2^{-j}\\
&{=} b^2+\delta^2 2^{2{s}-1}.
\end{align*}
It is easy to see that for small enough choice of the parameter $\delta>0$ the right hand side is bounded above by $B^2$ (the choice $\delta^2<(B^2-b^2)2^{1-2{s}}$ is sufficiently good) hence both $f_0$ and $f_{m,j}$ belong to the Sobolev ball $S^{{s}}(B)$. Then we show that $f_0$ satisfies the lower bound $\eqref{def: selfsim}$ as well. Similarly to the proof of Theorem \ref{Thm: InfLB} we have following from ${r}(1-\eps)<{s}$ that
\begin{align*}
\sum_{k=2^{{(1-\eps)J}}}^{2^J}f_{0,k}^2&\geq \sum_{k\in Z_{{\lceil(1-\eps)J\rceil}}^0}f_{0,k}^2= {(K_1^2/2)}2^{-2{r}{\lceil(1-\eps)J\rceil}}\\
&\geq 16\times2^{1+2{s}}B^2 2^{-2{s} J}\geq 16\times2^{1+2{s}}\|f_0\|_{{s},2}^2 2^{-2{s} J},\label{eq: selfsimf_m}
\end{align*}
for $J>J_0$ (where the parameter $J_0$ depends only on $r,s,B$ and $\eps$). The ${s}$-self-similarity of the functions $f_{m,j}$ follows exactly the same way.

Next we define the sequences $f_m$ ($m\in\mathcal{M}_j$) with the help of the sequences $f_{m,j}$, such that the $\ell^2$-distance between them is sufficiently large. It is easy to see that
\begin{align*}
\|f_{m,j,k}-f_{m',j,k}\|_2^2=2^{-j(2{s}+1)}\delta^2\sum_{k\in Z_{j}^1}(\beta_{m,j,k}-\beta_{m',j,k})^2.
\end{align*}
Then following from the Varshamov-Gilbert bound (\cite{TSY}) there exist a subset $\mathcal{M}_j\subset\{0,1\}^{|Z_j^1|}$ with cardinality $M_j=2^{2^{j}/16}$ such that
\begin{align*}
\sum_{k\in Z_{{j}}^1}(\beta_{m,j,k}-\beta_{m',j,k})^2\geq 2^{j}/16,
\end{align*}
for any $m\neq m'{\in\mathcal{M}_j}$. Therefore
\begin{align*}
\|f_{m,j}-f_{m',j}\|_2^2\geq (\delta^2/16) 2^{-2j{s}},
\end{align*}
for $m\neq m'\in\mathcal{M}_j$.
Then choosing $j=j_n$ such that {$j_n=\lfloor\log_2 n/(1+2s)\rfloor$} the $f_m\equiv f_{m,j_n}$ sequences are $2\times (\delta^2/2^{5}){n^{-2s/(1+2s)}}$ separated and are satisfying the self-similarity condition.

The KL-divergence is bounded by
\begin{align*}
K({\Pr}_{f_0},{\Pr}_{f_{m}})&=\frac{n}{2}\|f_{m}-f_0\|_2^2=\frac{n}{2}2^{-j_n(2{s}+1)}\delta^2\sum_{k\in Z_{j_n}^{1}}\beta_{m,j_n,k}^2\\
&\leq\frac{{2^{2s+1}}2^{j_n}\delta^2}{4}\leq \frac{{2^{2s+3}}\delta^2}{\ln 2} \ln M_{{j_n}}.
\end{align*}
Therefore we can conclude the proof by applying Theorem \ref{thm: help2} with $0<\delta<\sqrt{{{2^{-2s-4}}\ln 2}}$ (since in this case $\alpha={({2^{2s+3}}/\ln 2)}\delta^2<1/2$, hence the constant on the right hand side of $\eqref{eq: helpApp1}$ is positive) {and $r_n=(\delta^2/2^5)n^{-2{s}/(1+2{s})}$}.

\appendix
\section{}
We collect here some basic background material used in the proofs, most of which can be found or proved as in \cite{TSY} or \cite{GineNickl_book}.

\begin{theorem}\label{thm: AdaptiveEst}
Consider the Gaussian sequence model $\eqref{model}$ and assume that the true sequence $f\in\ell^2$ belongs to a collection of Sobolev balls $\cup_{s\in[s_{\min},s_{\max}]}S^{s}(B)$ for some fixed $0<s_{\min}<s_{\max}<\infty$ and (unknown) $B>0$. Then there exists a rate adaptive estimator $\hat{f}_n\in\ell^2$ over $\cup_{s\in[s_{\min},s_{\max}]}S^{s}(B)$, i.e., for every $s\in[s_{\min},s_{\max}]$, $B>0$
\begin{align*}
\sup_{f\in S^{s}(B)} \E_{f}\|\hat{f}_n-f\|_2\leq K(s) B^{1/(2s+1)}n^{-s/(2s+1)},
\end{align*}
where $0<K(s)<\infty$ is a fixed constant. We can moreover take $\hat f_n$ such that $\|\hat f_n\|_{s,2} = O_{{\Pr_f}}(1)$ uniformly in $f \in S^s(B)$.
\end{theorem}

\begin{theorem}\label{thm: help2}
Suppose $\mathcal{F}$ contains $\{f_m : m = 0, 1,...M\},~M>1,$  that are $2r_n$ separated
($d(f_m, f_{m'} ) \geq 2r_n,~ \forall m \neq m'$), and such that the $\Pr_{f_m}$ are all absolutely continuous with respect to $\Pr_{f_0}$ . Set $\bar{M} =
\max\{e,M\}$ and assume that for some $\alpha>0$
$$\frac{1}{M}\sum_{m=1}^{M}K({\Pr}_{f_m},{\Pr}_{f_0})\leq\alpha\log \bar{M}.$$
Then the minimax risk from is lower bounded by
$$\inf_{\tilde{f}_n}\sup_{f\in\mathcal{F}}\E_f d(\tilde{f}_n,f)\geq r_n\frac{\sqrt{\bar{M}}}{1+\sqrt{\bar{M}}}\Big(1-2\alpha-\sqrt{\frac{2\a}{\log\bar{M}}}\Big).$$
\end{theorem}

\begin{proposition}\label{prop: help1}
For the Gaussian vector $(y_k : k \in \mathbb{Z})$ from (\ref{model}) denote by ${\Pr}_{f}^{(n)}$ the product corresponding law on the cylindrical $\sigma$-algebra $\mathcal{C}$ of $\mathbb{R}^{\mathbb{Z}}$. If $(f_k:\, k \in \mathbb{Z})\in\ell^2$ then ${\Pr}_{f}^{(n)}$
is absolutely continuous with
respect to ${\Pr}_{0}^{(n)}$, and the likelihood ratio, for ${\Pr}_{0}^{(n)}$ is given by
\begin{equation}
\frac{d{\Pr}_{f}^{(n)}}{d {\Pr}_{0}^{(n)}}=\exp\Big\{n\sum_{k\in\mathbb{Z}}f_ky_k-\frac{n}{2}\|f\|_2^2\Big\}.\label{eq: helpApp1}
\end{equation}
\end{proposition}

\begin{theorem}\label{thm: help1}
Let $g_i,\, i = 1,...,n,$ be i.i.d. $N(0,{\sigma^2})$ and set $X =\sum_{i=1}^n(g_i^2-{\sigma^2}).$ Then for any $t\geq0$,
$${\Pr}_f(X >t)\leq {\exp\Big\{-\frac{t^2/\sigma^4}{4(n+t/\sigma^2)}\Big\}},$$
and the same inequality holds for $-X$.
\end{theorem}



\bibliographystyle{acm} 
\bibliography{references}





\end{document}